\documentclass{amsart}[12pt]

\usepackage{yfonts} %
\usepackage{amssymb} %
\usepackage{amsthm}
\usepackage{array}
\usepackage{booktabs}%
\usepackage{hhline}%
\usepackage{xy} %
\usepackage{epsfig}%
\usepackage{color}%
\usepackage{upgreek}
\usepackage[english]{babel}
\usepackage{epigraph}%
\usepackage{fancybox}%
\usepackage{shadow}
\usepackage{afterpage}
\usepackage{mathrsfs}
\usepackage{enumitem}
\usepackage{tabularx}
\usepackage{subcaption}
\usepackage{graphicx}
\usepackage{type1cm}
\usepackage{eso-pic}
\usepackage{color}
\usepackage{upgreek}
\usepackage[foot]{amsaddr}

\newtheorem{theorem}{Theorem}
\newtheorem{lemma}{Lemma}

\theoremstyle{definition}
\newtheorem{definition}{Definition}

\theoremstyle{plain}

\newcommand{\vt}{\vspace{.1cm}}
\newcommand{\vtt}{\vspace{.2cm}}
\newcommand{\R}{\mathbb{R} }

\newcommand{\C}{\mathbb{C} }
\newcommand{\N}{\mathbb{N} }
\newcommand{\h}{\mathbb{H}}

\newcommand{\hf}{\mathbb{H}_{\mathbb F}^m}

\newcommand{\s}{\mathbb{S}}

\renewcommand{\rho}{\varrho}

\renewcommand{\theta}{\varTheta}
\renewcommand{\Theta}{\varTheta}
\renewcommand{\Sigma}{\varSigma}
\renewcommand{\tau}{\uptau}
\usepackage{amsmath}

\newcommand{\overbar}[1]{\mkern 1.5mu\overline{\mkern-1.5mu#1\mkern-1.5mu}\mkern 1.5mu}

\newcommand{\transv}{\mathrel{\text{\tpitchfork}}}
\makeatletter
\newcommand{\tpitchfork}{%
  \vbox{
    \baselineskip\z@skip
    \lineskip-.52ex
    \lineskiplimit\maxdimen
    \m@th
    \ialign{##\crcr\hidewidth\smash{$-$}\hidewidth\crcr$\pitchfork$\crcr}
  }%
}
\makeatother

\begin{document}

\title[Translating Solitons in Riemannian Products]
{Translating Solitons to Flows by Powers of The Gaussian Curvature in Riemannian Products}
\author{Ronaldo Freire de Lima}
\address[A1]{Departamento de Matem\'atica - Universidade Federal do Rio Grande do Norte}
\email{ronaldo.freire@ufrn.br}
\maketitle

\begin{abstract}
  We consider translating solitons to flows by positive powers
  $\alpha$ of the Gaussian curvature  --- called $K^\alpha$-flows ---  in Riemannian products
  $M\times\mathbb R.$ We prove that, when  $M$ is the Euclidean space $\mathbb R^n,$ the
  sphere $\mathbb S^n,$ or one of the hyperbolic spaces $\mathbb{H}_{\mathbb F}^m,$ there exist complete
  rotational translating solitons to $K^\alpha$-flow in $M\times\mathbb R$ for certain values of $\alpha.$
\end{abstract}

\section{Introduction}
Over the last decades, the subject of extrinsic curvature flows in Riemannian manifolds
has experienced a significant development. Along this time,
special attention has been given to
mean curvature and   Gaussian curvature flows in Euclidean space,
resulting in achievements  such as the proof
of short time existence of solutions, and
of  their convergence (after rescaling)
to round spheres. (For a  thorough account of these and other extrinsic flows, we refer
to the recent book  \cite{andrewsetal}.) 

Likewise,  flows by powers
of the Gaussian curvature, which we call $K^\alpha$-\emph{flows}, have been
considered by many authors, most notably  B. Andrews
(see, e.g.,
\cite{andrews-chen,andrews-guan-li,andrews-maccoy,brendle-choi-dasdalopoulos,choietal,li-wang,urbas}).
Among other reasons, the interest in $K^\alpha$-flows comes from the fact that
they naturally relate  to a wide range of research fields,
varying from image processing to affine geometry and geometric analysis (cf. \cite{andrews}).

In the context of extrinsic curvature flows in Euclidean space, there are special hypersurfaces,
called \emph{translating solitons}, which have the distinguished property of moving
by translation under such a  flow. In fact, the construction and classification of complete
translating solitons constitutes a major problem in this theory (see, e.g.,
\cite{choietal01,martinetal,spruck-xiao,urbas}).
On this regard, we point out that
J. Urbas \cite{urbas} established, for all $\alpha\in(0,1/2],$
the existence of complete rotational translating solitons to the $K^\alpha$-flow in Euclidean space.

More generally, translating solitons can be defined in Riemannian products $M\times\R,$ where
$M$ is an arbitrary Riemannian manifold. In \cite{lira-martin},
assuming $M$  complete  with non-positive sectional curvature and rotationally invariant
metric, Lira and Martín obtained a one-parameter family of rotational translating
solitons to the  mean curvature flow in $M\times\R.$
Also, in \cite{ortegaetal,lawn-ortega},
results of the same nature were obtained in the context of Lorentzian  products.
On the other hand, to our knowledge, there are no works
on translating solitons to $K^\alpha$-flows in  products $M\times\R.$

In this paper, we prove existence of complete rotational
translating solitons to the
$K^\alpha$-flow in $M\times\R$ (for certain positive values of $\alpha$)
when $M$ is the Euclidean space
$\R^n,$ the sphere $\s^n,$ or one of the hyperbolic spaces
$\hf$ (rank-one symmetric spaces of noncompact type). In particular,
we recover the aforementioned  result of Urbas, giving it a new proof.

We remark that the translating solitons we construct here are all
bowl-type graphs of radial functions. In the Euclidean and hyperbolic cases,
these graphs are entire, whereas
in the spherical case they project onto an open hemisphere $B$ of $\s^{n},$
being asymptotic to the  half-cylinder $\partial B\times[0,+\infty).$

The paper is organized as follows. In Section \ref{sec-preliminaries},
we fix  notation regarding hypersurfaces of Riemannian products $M\times\R.$ In Section
\ref{sec-graphs}, we discuss graphs defined on families of parallel hypersurfaces,
setting some formulae. In Section \ref{sec-Ksolitons}, we introduce the
$K^\alpha$-flow in $M\times\R$ and establish a key lemma. Finally, in
Section \ref{sec-mainresult}, we state and prove our main result.

\section{Preliminaries} \label{sec-preliminaries}
Given a Riemannian manifold $M^n,$ $n\ge 2,$  consider the Riemannian product
$M\times\R$ endowed with its standard metric
\[
\langle\,,\,\rangle=\langle\,,\,\rangle_{M}+dt^2,
\]
and denote its Levi-Civita connection by $\overbar\nabla.$

Let $\Sigma^n$ be an oriented hypersurface of  $M\times\R.$ Set
$N$ for its unit normal field and   $A$ for its shape operator  with respect to
$N,$ so that
\[
AX=-\overbar\nabla_XN,  \,\, X\in T\Sigma,
\]
where  $T\Sigma$ stands for  the tangent bundle of $\Sigma$.
The principal curvatures of $\Sigma,$ that is,
the eigenvalues of the shape operator  $A,$ will be denoted by $k_1\,, \dots ,k_n$.
In this setting, the \emph{Gaussian curvature} of $\Sigma$ is the function $K:\Sigma\rightarrow\R$ defined as
\[
K:=\det A=k_1\dots k_n\,.
\]

We denote by $\partial_t$ the  gradient  of the
projection $\pi_{\scriptscriptstyle \R}$ of $M\times\R$  on the factor $\R.$
Given a hypersurface  $\Sigma\subset M\times\R,$ its
\emph{height function} $\xi$ and its \emph{angle function} $\Theta$
are defined by the following identities:
\[
\xi(x):=\pi_{\scriptscriptstyle \R}|_\Sigma \quad\text{and}\quad \Theta(x):=\langle N(x),\partial_t\rangle, \,\, x\in\Sigma.
\]

We shall denote the gradient of $\xi$ on $\Sigma$ by $T,$ so that
\begin{equation} \label{eq-gradxi}
T=\partial_t-\Theta N.
\end{equation}

Given $t\in\R,$ the set  $P_t:=M\times\{t\}$ is called a \emph{horizontal hyperplane}
of $M\times\R.$ Horizontal hyperplanes are all isometric to $M$ and totally geodesic in
$M\times\R.$ In this context,
we call a transversal intersection $\Sigma_t:=\Sigma\transv P_t$ a \emph{horizontal section}
of $\Sigma.$ Any horizontal section $\Sigma_t$ is a hypersurface of $P_t$\,.
So, at any point $x\in\Sigma_t\subset\Sigma,$ the tangent space $T_x\Sigma$ of $\Sigma$ at $x$ splits
as the orthogonal sum
\begin{equation}\label{eq-sum}
T_x\Sigma=T_x\Sigma_t\oplus {\rm Span}\{T\}.
\end{equation}

\begin{definition}
A hypersurface  $\Sigma\subset M\times\R$ is called \emph{rotational}, if there exists
a fixed point $o\in M$ such that any connected component of
any horizontal section $\Sigma_t$ of $\Sigma$ is contained in a geodesic  sphere  of $M\times\{t\}$
with center at $o\times\{t\}.$
If so, the set $\{o\}\times\R$ is called the \emph{axis} of $\Sigma.$
\end{definition}

Recall that the rank-one symmetric spaces of non-compact type
are the hyperbolic spaces $\h_\R^m, \,\,\h_\C^m, \,\, \h_{\mathbb K}^m$ {and} $ \h_{\mathbb O}^2,$ $m\ge 1,$
called \emph{real hyperbolic space, complex hyperbolic space, quaternionic hyperbolic space} and
\emph{Cayley hyperbolic plane,} respectively.
We will adopt the unified notation $\hf$ for these  spaces with $m=2$ for $\mathbb F=\mathbb O.$
The real dimension  of $\hf$ is $n=m\dim \mathbb F.$

We remark that the hyperbolic spaces $\hf$ are all Einstein--Hadamard manifolds
whose sectional curvatures are negative and pinched. In particular, their
geodesic spheres are all strictly convex (see, e.g., \cite{berndtetal}).

\section{Graphs on  Parallel Hypersurfaces}  \label{sec-graphs}
Consider
an isometric immersion
\[f:M_0^{n-1}\rightarrow M^n\]
between two Riemannian manifolds $M_0^{n-1}$ and $M^n,$
and suppose that there is a neighborhood $\mathscr{U}$
of $M_0$ in $TM_0^\perp$  without focal points of $f,$ that is,
the restriction of the normal exponential map $\exp^\perp_{M_0}:TM_0^\perp\rightarrow M$ to
$\mathscr{U}$ is a diffeomorphism onto its image. In this case, denoting by
$\eta$ the unit normal field  of $f,$   there is an open interval $I\owns 0$
such that, for all $p\in M_0,$ the curve
\begin{equation}\label{eq-geodesic}
\gamma_p(s)=\exp_{\scriptscriptstyle M}(f(p),s\eta(p)), \, s\in I,
\end{equation}
is a well defined geodesic of $M$ without conjugate points. Thus,
for all $s\in I,$
\[
\begin{array}{cccc}
f_s: & M_0 & \rightarrow & M\\
     &  p       & \mapsto     & \gamma_p(s)
\end{array}
\]
is an immersion of $M_0$ into $M,$ which is said to be \emph{parallel} to $f.$
Observe that, given $p\in M_0$, the tangent space $f_{s_*}(T_p M_0)$ of $f_s$ at $p$ is the parallel transport of $f_{*}(T_p M_0)$ along
$\gamma_p$ from $0$ to $s.$ We also remark that,  with the induced metric,
the unit normal  $\eta_s$  of $f_s$ at $p$ is given by
\[\eta_s(p)=\gamma_p'(s).\]

\begin{definition}
Let $\phi:I\rightarrow \phi(I)\subset\R$ be an increasing diffeomorphism, i.e., $\phi'>0.$
With the above notation, we call the set
\begin{equation}\label{eq-paralleldescription1}
\Sigma:=\{(f_s(p),\phi(s))\in M\times\R\,;\, p\in M_0, \, s\in I\},
\end{equation}
the \emph{graph} determined by $\{f_s\,;\, s\in I\}$ and $\phi,$ or $(f_s,\phi)$-\emph{graph}, for short.
\end{definition}

For an arbitrary  point $x=(f_s(p),\phi(s))$ of an
$(f_s,\phi)$-{graph} $\Sigma,$ one has
\[T_x\Sigma=f_{s_*}(T_p M_0)\oplus {\rm Span}\,\{\partial_s\}, \,\,\, \partial_s=\eta_s+\phi'(s)\partial_t.\]
So, a  unit normal  to $\Sigma$ is
\begin{equation} \label{eq-normal}
N=\frac{-\phi'}{\sqrt{1+(\phi')^2}}\eta_s+\frac{1}{\sqrt{1+(\phi')^2}}\partial_t\,.
\end{equation}
In particular, its  angle function  is
\begin{equation} \label{eq-thetaparallel}
\Theta=\frac{1}{\sqrt{1+(\phi')^2}}\,\cdot
\end{equation}

As shown in  \cite[Theorem 6]{delima-roitman}), any  $(f_s,\phi)$-graph $\Sigma$
has the $T$-\emph{property}, meaning that
$T$ is a principal direction at any point of $\Sigma$.
More precisely,
\begin{equation}\label{eq-principaldirection}
AT=\frac{\phi''}{(\sqrt{1+(\phi')^2})^3}T.
\end{equation}

Conversely, any hypersurface of $M\times\R$ with non vanishing angle function having $T$ as a principal direction
is given locally as an $(f_s,\phi)$-graph.

Given an
$(f_s,\phi)$-graph $\Sigma,$
let $\{X_1\,,\dots ,X_n\}$ be
the orthonormal frame of principal directions of $\Sigma$
in which $X_n=T/\|T\|.$ In this case,
for $1\le i\le n-1,$
the fields $X_i$  are all horizontal, that is, tangent to $M$ (cf. \eqref{eq-sum}).
Therefore, setting
\begin{equation}\label{eq-rho}
\rho:=\frac{\phi'}{\sqrt{1+(\phi')^2}}
\end{equation}
and considering \eqref{eq-normal}, we have, for all $i=1,\dots ,n-1,$ that
\[
k_i=\langle AX_i,X_i\rangle=-\langle\overbar\nabla_{X_i}N,X_i\rangle=\rho\langle\overbar\nabla_{X_i}\eta_s,X_i\rangle=-\rho k_i^s,
\]
where $k_i^s$ is the $i$-th principal curvature of $f_s\,.$ Also,
it follows from \eqref{eq-principaldirection}
that $k_n=\rho'.$ Thus, the principal curvatures of the $(f_s,\phi)$-graph $\Sigma$
at $(f_s(p),\phi(s))\in\Sigma$ are
\begin{equation}\label{eq-principalcurvatures}
k_i=-\rho(s) k_i^s(p) \,\, (1\le i\le n-1) \quad\text{and}\quad k_n=\rho'(s).
\end{equation}
In particular, the Gaussian curvature $K$ of $\Sigma$ at $(f_s(p),\phi(s))$ is given by
\begin{equation} \label{eq-curvaturegraph}
K=(-\rho(s))^{n-1}K_s(p)\rho'(s)\,,
\end{equation}
where $K_s$ denotes the Gaussian curvature of
the hypersurface $f_s$\,.

We remark that  the function $\rho$ defined in \eqref{eq-rho} determines
the function $\phi.$
Indeed, it follows from equality \eqref{eq-rho} that
\begin{equation}\label{eq-phi10}
\phi(s)=\int_{s_0}^{s}\frac{\rho(u)}{\sqrt{1-\rho^2(u)}}du+\phi(s_0), \,\,\, s_0, s\in I.
\end{equation}

In addition, from \eqref{eq-thetaparallel} and \eqref{eq-rho},
the unit normal $N$ defined in \eqref{eq-normal} can be
written as $N=-\rho\eta_s+\theta\partial_t$\,. Hence, the relation
\begin{equation} \label{eq-theta2}
\theta=\sqrt{1-\rho^2}
\end{equation}
holds everywhere on any $(f_s,\phi)$-graph $\Sigma.$
In particular, one has $\rho=\|T\|.$

Now, let us assume that the family of parallel hypersurfaces
\[
\mathscr F:=\{f_s:M_0\rightarrow M\,;\, s\in I\}
\]
is  \emph{isoparametric}, that is,
for each $s\in I,$ any principal curvature $k_i^s$  of $f_s\in\mathscr F$ is  constant (possibly depending on $i$ and $s$).
We shall  assume further that each $f_s$ is strictly convex with $k_i^s<0.$

In this setting, the Gaussian curvature $K_s$ of each hypersurface $f_s$ is a (non vanishing) function of $s$ alone.
Hence, from \eqref{eq-curvaturegraph}, the same is true for any $(f_s,\phi)$-graph $\Sigma$ built on
$\mathscr F,$ that is,
\begin{equation}\label{eq-curvaturegraph2}
K(f_s(p),\phi(s))=(-\rho(s))^{n-1}\rho'(s)K_s\,,  \,\,\, (f_s(p),\phi(s))\in\Sigma.
\end{equation}

It should also be noticed that, by equalities \eqref{eq-principalcurvatures},
such a graph $\Sigma$ is strictly convex.

\section{Translating $K^\alpha$-Solitons in $M\times\R$}  \label{sec-Ksolitons}

Given a real number $\alpha\ne 0,$ we say that an oriented
strictly convex hypersurface $\Sigma$ of $M\times\R$
\emph{moves under} $K^\alpha$-flow if there exists a one-parameter
family of immersions $F\colon M_0\times[0,u_0)\rightarrow M\times\R,$ $u_0\le+\infty,$
such that
\begin{equation} \label{eq-Kalphaflow}
\left\{
\begin{array}{l}
\frac{\partial F}{\partial u}^\perp(p,u)=K^\alpha(p,u)N(p,u).\\[1ex]
F(M_0\,, 0)=\Sigma,
\end{array}
\right.
\end{equation}
where $N(p,u)$ is the inward 
unit normal to the hypersurface $F_u:=F(.\,, u)$,
$K(p,u)$ is the Gaussian curvature
of $F_u$ with respect to $N_u:=N(.\,, u),$ and
$\frac{\partial F}{\partial u}^\perp$ denotes the normal component of
$\frac{\partial F}{\partial u},$  i.e.,
\[
\frac{\partial F}{\partial u}^\perp=\left\langle\frac{\partial F}{\partial u},N_u\right\rangle N_u\,.
\]

In particular, the first equality in \eqref{eq-Kalphaflow} is equivalent to
\begin{equation}  \label{eq-condition}
\left\langle\frac{\partial F}{\partial u}(p,u),N(p,u)\right\rangle=K^\alpha(p,u).
\end{equation}
We call such a map $F$  a $K^\alpha$-\emph{flow} in $M\times\R.$

Denote by $\exp$ the exponential map of $M\times\R$ and
consider an isometric immersion $F_0\colon M_0\rightarrow M\times\R$.
Define then the map
\[
F(p,u):=\exp_{F_0(p)}(u\partial_t), \,\,(p,u)\in M_0\times [0,+\infty),
\]
and notice that, for each $u\in (0,+\infty),$ the hypersurface  $F(M_0\,, u)$ is nothing
but a vertical translation of $\Sigma:=F(M_0,0).$ Since vertical translations are isometries of
$M\times\R$ keeping the second factor invariant,
we have that $\Sigma$ and $F(M_0\,, u)$
are congruent with coinciding angle functions and Gaussian curvature, that is,
\begin{equation} \label{eq-angles}
\theta(p,u)=\theta(p,0) \quad\text{and}\quad K(p,u)=K(p,0)\,\,\, \forall (p,u)\in M_0\times [0,u_0).
\end{equation}

Now, differentiating $F$ with respect to $u,$  we have
\begin{equation} \label{eq-partialt}
\frac{\partial F}{\partial u}(p,u)=(d\exp_{F_0(p)})(u\partial_t)\partial_t=\partial_t\,.
\end{equation}

From \eqref{eq-angles} and \eqref{eq-partialt}, we have that
$F$ satisfies \eqref{eq-condition} if and only if
the  equality
\[
\theta (p,0)=K^\alpha(p,0)
\]
holds on $M_0$\,. This fact motivates the following concept.

\begin{definition}
Given  $\alpha> 0,$ we say that an oriented
strictly convex hypersurface $\Sigma$ of the
Riemannian product $M\times\R$
is a \emph{translating soliton to the $K^\alpha$-curvature flow} (or simply
a \emph{translating $K^\alpha$-soliton}) if the equality
\[
K^\alpha=\theta
\]
holds everywhere on $\Sigma.$ (From the above discussion, any such
hypersurface $\Sigma$ is the initial data of a $K^\alpha$-flow by vertical translations.)
\end{definition}

Let us consider now an $(f_s,\phi)$-graph $\Sigma$
in $M\times\R$ such that the family
\[
\mathscr F:=\{f_s:M_0\rightarrow M\,;\, s\in I\}
\]
is isoparametric with $k_i^s<0.$  In this case,
it follows from \eqref{eq-theta2} and \eqref{eq-curvaturegraph2} that
$\Sigma$ is a translating $K^\alpha$-soliton if and only if its associated
$\rho$ function satisfies 
\begin{equation} \label{eq-edo0}
(-\rho(s))^{n-1}\rho'(s)K_s=(1-\rho^2(s))^{\frac{1}{2\alpha}}, \,\,\, s\in I.
\end{equation}

Since $K_s$ never vanishes on $I,$  we have that
\eqref{eq-edo0} holds for $\rho$ if and only if
\begin{equation} \label{eq-odetau}
\tau'(s)=\frac{(-1)^{n-1}n(1-\tau^{\frac2n})^{\frac{1}{2\alpha}}}{K_s}\,, \quad \tau:=\rho^n.
\end{equation}

Summarizing, we have the following result.
\begin{lemma} \label{lem-parallel}
Let $\Sigma$ be an $(f_s,\phi)$-graph in $M\times\R$ whose associated
family
\[\mathscr F:=\{f_s:M_0\rightarrow M\,;\, s\in I \}\]
of oriented parallel  hypersurfaces  is isoparametric with each of them having negative principal curvatures.
Then, $\Sigma$ is a translating $K^\alpha$-soliton if and only if  \eqref{eq-odetau}
holds for $\tau:=\rho^{n},$
where $\rho:I\rightarrow\R$ is as in \eqref{eq-rho}.
\end{lemma}

\section{The Main Result} \label{sec-mainresult}
For $n\ge 2$ and $m\ge 1,$ let $M$ be one of the following manifolds, each of them endowed with its canonical
Riemannian metric:  Euclidean space $\R^n,$ the sphere $\s^n,$ or one of the hyperbolic spaces $\hf.$
Then, define
\begin{equation} \label{eq-radius}
\mathcal R_M:=\left\{
\begin{array}{ccc}
+\infty & \text{if} & M\ne\s^n\\[1ex]
\pi/2   & \text{if} & M=\s^n
\end{array}
\right.
\end{equation}
and consider a family
\begin{equation} \label{eq-Fgeodesicspheres}
\mathscr F:=\{f_s:\s^{n-1}\rightarrow M\,; \,\,\, s\in (0,\mathcal R_M)\}
\end{equation}
of  concentric geodesic spheres of $M$ indexed by their radiuses, that is,
for a fixed point $o\in M$, and for each $s\in (0,\mathcal R_M),$
$f_s(\s^{n-1})$ is the (strictly convex) geodesic sphere $S_s(o)$ of $M$ with center at $o$ and radius $s.$
In accordance to the notation of Section \ref{sec-graphs}, for each $s\in (0,\mathcal R_M),$
we  choose the outward orientation of $f_s$\,,
so that \emph{any principal curvature $k_i^s$ of \,$f_s$ is negative}.

For $M=\hf,$ 
the principal curvatures $k_i^s$ of the geodesic spheres $f_s\in\mathscr F$ are:
\begin{equation} \label{eq-princcurvhf}
\begin{aligned}
    k_1^s &= -\frac{1}{2}\coth(s/2) \,\,\,\text{with multiplicity}\,\,\,  n-p-1. \\[1ex]
    k_2^s &=-\coth(s) \,\,\, \text{with multiplicity} \,\,\,  p,
  \end{aligned}
\end{equation}
where $n=\dim\hf,$ $p=n-1$ for $\h^n$,  $p=1$ for $\h_\C^m$, $p=3$ for $\h_{\mathbb K}^m$, and
$p=7$ for $\h_{\mathbb O}^2$  (see, e.g., \cite[pgs. 353, 543]{cecil-ryan} and \cite{kimetal}).
Thus, the Gaussian curvature $K_s$ of the geodesic sphere $S_s(o)$ of $\hf$ is given by
\[
K_s=(-1)^{n-1}\left(\frac{1}{2}\coth(s/2)\right)^{n-p-1}(\coth s)^p\,.
\]
In particular, the function
\[
\frac{(-1)^{n-1}}{K_s}=(2\tanh(s/2))^{n-p-1}(\tanh s)^p
\]
is well defined and nonnegative  on $[0,+\infty).$

As is well known, in the cases $M=\R^n$ and $M=\s^n$ one has
\[
\frac{(-1)^{n-1}}{K_s}=s^{n-1} \quad\text{and}\quad \frac{(-1)^{n-1}}{K_s}=(\tan s)^{n-1},
\]
respectively.

It is easily seen that, in any of these three cases, the equality
\begin{equation} \label{eq-limitintegral}
\lim_{s\rightarrow\mathcal R_M}\int_{0}^{s}\frac{(-1)^{n-1}}{K_v}dv=+\infty.
\end{equation}
holds.
Finally, set
\begin{itemize}[parsep=2ex]
\item $\delta_n:=\,\,\left\{
\begin{array}{ccc}
1/2 & \text{if} & n=2.\\[1ex]
\max\{\frac14,\frac{1}{n-1}\}  & \text{if} & n>2.\\[1ex]
\end{array}
\right.
$

\item $
\mathcal I_M:=\left\{
\begin{array}{lcc}
(0,1/2] & \text{if} & M\ne\s^n.\\[1ex]
[\delta_n,1/2]   & \text{if} & M=\s^n.
\end{array}
\right.
$
\end{itemize}
\vt

With this notation, we now state and prove  our main result.

\begin{theorem} \label{th-solitoHn}
Let $M$ and $\mathcal I_M\subset (0,1/2]$ be as above.
Then, for all $\alpha\in\mathcal I_M,$ there exists a complete rotational
strictly convex translating $K^\alpha$-soliton in the closed half-space $M\times[0,+\infty),$
whose height function is unbounded. In addition, the following assertions hold:
\begin{itemize}[parsep=1ex]
    \item For $M\ne\s^n,$ $\Sigma$ is  an entire graph over $M.$
    \item For $M=\s^n,$ $\Sigma$ is a graph over an open hemisphere $B\subset\s^n,$
          being  asymptotic to the half-cylinder $\partial B\times[0,+\infty).$
\end{itemize}
\end{theorem}

\begin{proof}
Let $\mathscr F$ be a family of parallel geodesic spheres
of $M$ as in \eqref{eq-Fgeodesicspheres}. We intend to
determine a function $\phi\in C^{\infty}[0,\mathcal R_M)$ such that
the corresponding $(f_s,\phi)$-graph $\Sigma$ becomes the desired
$K^\alpha$-soliton. With this purpose, let us consider the equality
\eqref{eq-odetau} as an ODE with variable $\tau,$ and rewrite it as
\begin{equation} \label{eq-edotau2}
\frac{d\tau}{(1-\tau^{\frac2n})^{\frac{1}{2\alpha}}}=\frac{(-1)^{n-1}n}{K_s}ds.
\end{equation}

In order to solve \eqref{eq-edotau2}, consider  the functions
\[
\Phi\colon [0,1)\rightarrow [0,+\infty)\quad\text{and}\quad \Psi\colon [0,\mathcal R_M)\rightarrow [0,+\infty)
\]
given by
\[
\Phi(\tau):=\int_{0}^{\tau}\frac{du}{(1-u^{\frac2n})^{\frac{1}{2\alpha}}}\qquad\text{and}\qquad \Psi(s):=\int_{0}^{s}\frac{(-1)^{n-1}n}{K_v}dv.
\]

It follows from  \eqref{eq-limitintegral} that $\Psi$ is a diffeomorphism.
Let us show that the same is true for $\Phi.$
Indeed, $\Phi$   is a $C^\infty$ function with positive derivative,
which implies that it is a diffeomorphism over its image. So, it remains to prove
that $\Phi([0,1))=[0,+\infty).$ To this end, notice first that
$1-u^{2/n}\le 1-u\,\forall u\in[0,1).$ Therefore, setting $p:=1/(2\alpha)\ge 1,$
one has
\[
\Phi(\tau)\ge\int_{0}^{\tau}\frac{du}{(1-u)^{p}}=
\left\{
\begin{array}{ccc}
  \log\left(\frac{1}{1-\tau}\right) & \text{if} & p=1. \\[2ex]
  \frac{(1-\tau)^{1-p}-1}{p-1} & \text{if} & p> 1,
\end{array}
\right.
\]
which implies that $\Phi(\tau)\rightarrow +\infty$ as
$\tau\rightarrow 1.$ Hence, $\Phi([0,1))=[0,+\infty),$ as asserted.

Now, we can define $\tau:[0,\mathcal R_M)\rightarrow\R$ by
\begin{equation}\label{eq-taugeral}
  \tau(s)=\Phi^{-1}(\Psi(s)),
\end{equation}
which is clearly a solution of \eqref{eq-edotau2} (and so of \eqref{eq-odetau})  satisfying
\begin{equation} \label{eq-taulimit1}
0=\tau(0)\le\tau(s)<1 \,\,\,\, \forall  s\in[0,\mathcal R_M) \quad\text{and}\quad \lim_{s\rightarrow\mathcal R_M}\tau(s)=1.
\end{equation}

Therefore, by Lemma \ref{lem-parallel}, the corresponding $(f_s,\phi)$-graph $\Sigma$ is a translating $K^\alpha$-soliton
in $M\times\R.$ Moreover, since $\rho={\tau}^{1/n}$, we have from \eqref{eq-phi10} (with $s_0=0$) and \eqref{eq-taulimit1}
that the  function $\phi$  is defined in $[0,\mathcal R_M)$ and satisfies
\[
\phi(0)=\phi'(0)=0 \quad\text{and}\quad \lim_{s\rightarrow\mathcal R_M}\phi'(s)=+\infty.
\]

We conclude from the above discussion that, if $M\ne\s^n$,
then $\Sigma$ is an unbounded entire graph over $M$  which is
contained in the closed half-space $M\times[0,+\infty).$
In particular, $\Sigma$ is complete (Fig. \ref{fig-soliton1}).

\begin{figure}[htbp]
\includegraphics{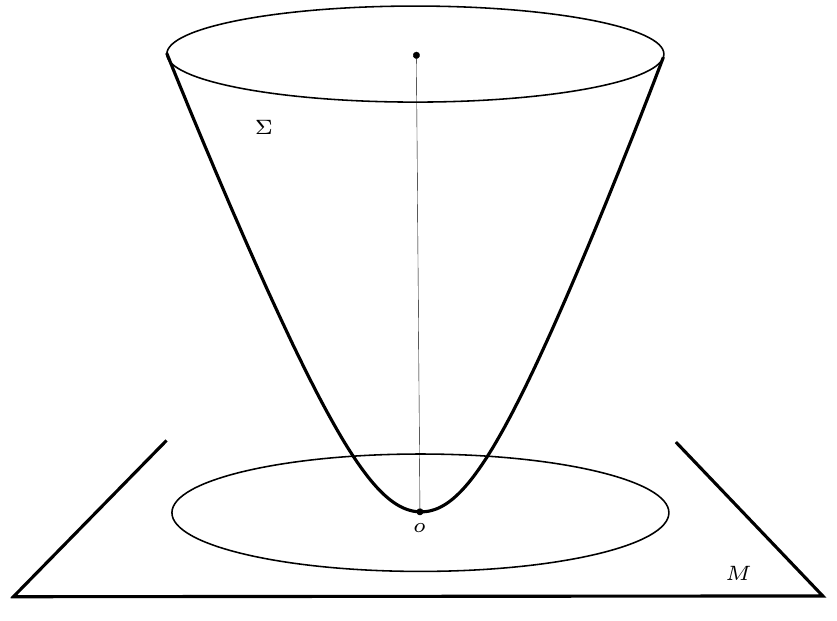}
\caption{\small A  translating $K^\alpha$-soliton in $M\times\R.$}
\label{fig-soliton1}
\end{figure}

If  $M=\s^n$, then $\Sigma$ is a graph over the open hemisphere
$B$ centered at $o\in\s^n$ (Fig. \ref{fig-soliton2}).
Thus, to conclude that $\Sigma$ is complete and
asymptotic to $\partial B\times[0,+\infty),$ we must
prove that $\phi$ is unbounded.

For  $n=2,$ we have from the hypothesis that $\alpha=1/2.$ In this case, it is easily checked that
$\rho(s)=\sin s$ is the solution of \eqref{eq-edo0}
satisfying $\rho(0)=0.$ So, from \eqref{eq-phi10}, 
\[
\phi(s)=\int_{0}^{s}\frac{\rho(u)}{\sqrt{1-\rho^2(u)}}du=\int_{0}^{s}\tan (u)du=\log\left(\frac{1}{\cos s} \right), \,\,\,\, s\in(0,\pi/2),
\]
which implies that $\phi(s)\rightarrow+\infty$ as $s\rightarrow\pi/2.$ 

For $n>2,$ let us show first
that  $\tau'$ is bounded in $(0,\pi/2).$
Indeed, from \eqref{eq-odetau}, 
$$\tau'(s)=n(\tan s)^{n-1}(1-(\tau(s))^{2/n})^{p}, \quad p=\frac{1}{2\alpha}\cdot$$
So, setting  
\[
\mu_1(s):=n(\tan s)^{n-2}(\sec s)^{2}(1-(\tau(s))^{2/n})^p \quad\text{and}\quad \mu_2(s):=(1-(\tau(s))^{\frac2n})\tan s,
\]
a straightforward calculation  yields 
\begin{eqnarray} \label{eq-tau''}
\tau''(s) & = & \mu_1(s)\left[(n-1)-\frac{2p(\tau(s))^{\frac{2-n}{n}}(\cos s)^2(\tan s) \tau'(s)}{n(1-(\tau(s))^{\frac2n})} \right]\nonumber\\
          & = & \mu_1(s)\left[(n-1)-\frac{2p(\tau(s))^{\frac{2-n}{n}}(\sin s)^2 \tau'(s)}{n\mu_2(s)}\right], \,\,\, s\in(0,\pi/2).
\end{eqnarray}

Let us suppose, by contradiction, that there exists a sequence $(s_k)$ in
$(0,\pi/2)$ such that $\tau'(s_k)\rightarrow+\infty.$ We can assume, without loss of generality, that
$\tau''(s_k)>0$ for all $k\in\N.$ Under this assumption, we have from \eqref{eq-tau''} that
the sequence $\mu_2(s_k)$ is unbounded above.
Otherwise, for a sufficiently large $k,$ the expression in the brackets
would be negative for $s=s_k,$ which would give $\tau''(s_k)<0$  --- a contradiction. 

\begin{figure}[htbp]
\includegraphics{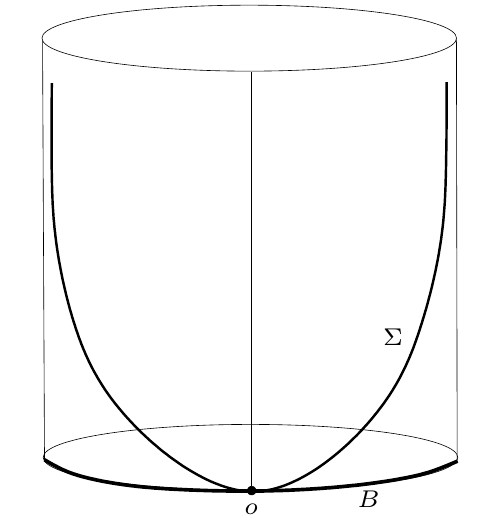}
\caption{\small A translating $K^\alpha$-soliton in $\s^n\times\R.$}
\label{fig-soliton2}
\end{figure}

So, we can assume $\mu_2(s_k)\rightarrow+\infty.$
However, from our choice of $\alpha,$ we have that $0\le p-1\le 1,$  which yields 
$(1-(\tau(s_k))^{2/n})^{p-1}\ge (1-(\tau(s_k))^{2/n}).$ Thus, 
\[
\frac{\tau'(s_k)}{\mu_2(s_k)}=n(\tan s_k)^{n-2}(1-(\tau(s_k))^{2/n})^{p-1}\ge n(\tan s_k)^{n-3}\mu_2(s_k)\rightarrow+\infty,
\]
which, together with equality \eqref{eq-tau''}, gives $\tau''(s_k)<0$ for a sufficiently large $k$ --- again a contradiction.
Therefore, $\tau'$ (and so  $\rho'$) is  bounded in $(0,\pi/2).$

Now, let  $C>0$  be such that $\rho'(s)<C$ for all $s\in (0,\pi/2).$
Since we are assuming $\alpha(n-1)\ge 1,$  equalities \eqref{eq-phi10} and  \eqref{eq-edo0} yield
\[
\phi(s)=\int_{0}^{s}\frac{\rho(u)}{\sqrt{1-\rho^2(u)}}du=\int_{0}^{s}\frac{(\tan u)^{\alpha(n-1)}}{(\rho(u))^{\alpha(n-1)-1}(\rho'(u))^\alpha}du\ge
\frac{1}{C^\alpha}\int_{0}^{s}\tan u\, du,
\]
which implies that $\phi(s)\rightarrow+\infty$ as $s\rightarrow\pi/2.$
This concludes the proof.
\end{proof}

We finish by noting that, for a given  Riemannian manifold $M,$ we have
from Lemma \ref{lem-parallel} that there exist local
translating $K^\alpha$-solitons in
$M\times\R$ as long as $M$ admits families of isoparametric,
strictly convex hypersurfaces. This applies, for instance, to
the hyperbolic spaces $\hf$ and their  families of
parallel horospheres (see Proposition-(vi), pg 88, in \cite{berndtetal}), to
the real hyperbolic space $\h_\R^n$ and its families
of equidistant hypersurfaces to
a given totally geodesic hyperplane, as well as  to
$\s^n$ and any of its many  families
of strictly convex isoparametric hypersurfaces (cf. \cite{dominguez-vazquez} and the references therein).

\vtt
\noindent
\emph {Acknowledgements.} The author is indebt to Isabel L. Rios and Roberto
T. de Oliveira for helpful conversations during the preparation of this paper.

\end{document}